\documentclass[11pt]{amsart}
\usepackage[latin1]{inputenc}
\usepackage{pdfsync}
\usepackage[english]{babel}

\usepackage[colorlinks=true,urlcolor=blue,linkcolor=black,citecolor=blue]{hyperref}
\usepackage{graphicx}

\usepackage[capitalize]{cleveref}
\usepackage{fullpage}
\usepackage{cite}
\usepackage{color}
\usepackage{amsmath}
\usepackage{amssymb}
\usepackage{amsfonts}
\usepackage{mathrsfs}
\usepackage{t1enc , graphicx}
\usepackage{verbatim}
\usepackage{bbm}
\usepackage{enumerate}

\usepackage[left= 1.5 in, right= 1.5 in ,top= 1 in, bottom = 2 in]{geometry}
\newcommand{\R}{\mathbb{R}}

\newcommand{\C}{\mathbb{C}}

\newcommand{\Q}{\mathbb{Q}}
\newcommand{\Z}{\mathbb{Z}}
\newcommand{\N}{\mathbb{N}}
\newcommand{\E}{\mathbb{E}}
\renewcommand{\P}{\mathbb{P}}

\renewcommand{\mod}[1]{\text{ (mod #1)}}
\newcommand{\sse}{\subseteq}

\newcommand{\EC}{\underset{n \in [N]\,}\E}
\newcommand{\ECC}[2]{\underset{#1 \in #2\,}{\E}}

\newcommand{\Elog}{\underset{n \in [N]\,}{\E^{\log}}}
\newcommand{\Elogg}[2]{\underset{#1 \in #2\,}{\E^{\log}}}

\newcommand{\AVG}{\frac{1}{N} \sum_{n=1}^{N}}
\newcommand{\AVGK}{\frac{1}{K} \sum_{k=1}^{K}}

\newtheorem{theorem}{Theorem}[section]
\newtheorem*{theorem*}{Theorem}
\newtheorem{prop}[theorem]{Proposition}
\newtheorem{lemma}[theorem]{Lemma}

\newtheorem{corollary}[theorem]{Corollary}
\newtheorem*{corollary*}{Corollary}

\theoremstyle{definition}
\newtheorem*{definition*}{Definition}

\theoremstyle{remark}

\newtheorem{remark}[theorem]{Remark}
\newtheorem*{remark*}{Remark}

\title{A Dynamical Approach to the Asymptotic Behavior of the Sequence \(\Omega(n)\)}

\author{Kaitlyn Loyd}

\address{Northwestern University\\
        Evanston, IL \\ 
        60208}
        
\email{\href{mailto:loydka@math.northwestern.edu}{loydka@math.northwestern.edu}}

\date{September 17, 2021}

\thanks{The author was partially supported by NSF grant DMS-1502632.}

\begin{document}
\maketitle

\begin{abstract}
    We study the asymptotic behavior of the sequence \( \{\Omega(n) \}_{ n \in \N } \) from a dynamical point of view, where \( \Omega(n) \) denotes the number of prime factors of \( n \) counted with multiplicity. First, we show that for any non-atomic ergodic system $(X, \mathcal{B}, \mu, T)$, the operators $T^{\Omega(n)}: \mathcal{B} \to L^1(\mu)$ have the strong sweeping-out property. In particular, this implies that the Pointwise Ergodic Theorem does not hold along $\Omega(n)$. Second, we show that the behaviors of $\Omega(n)$ captured by the Prime Number Theorem and Erd\H{o}s-Kac Theorem are disjoint, in the sense that their dynamical correlations tend to zero.  
\end{abstract}

\section{Introduction}
\label{sec: Introduction}
    
    For $n \in \N$, let $\Omega(n)$ denote the number of prime factors of $n$, counted with multiplicity. The study of the asymptotic behavior of \( \Omega(n) \) has a rich history and finds important applications to number theory. For instance, the classical Prime Number Theorem is equivalent to the statement that the set \( \{ n \in \N : \Omega(n) \text{ is even}\} \) has asymptotic density 1/2 \cite{Landau1953, Mangoldt1897}. Recently, a dynamical approach to this question was introduced by Bergelson and Richter \cite{BR2020}. They show that given a uniquely ergodic dynamical system $(X, \mu, T)$, the sequence $\{T^{\Omega(n)} x\}_{n \in \N}$ is uniformly distributed in $X$ for every point $x \in X$ (see \cref{sec: Background} for relevant definitions). The precise statement is as follows:
    
    \begin{theorem}[Theorem A in \cite{BR2020}]
    \label{thm: Bergelson-Richter Theorem}
        Let $(X, \mu, T)$ be uniquely ergodic. Then  
        \[
            \lim_{N \to \infty} \AVG g(T^{\Omega(n)} x) = \int_X g \, d\mu
        \]
        for all $x \in X$ and $g \in C(X)$.
    \end{theorem}

    The goal of this paper is to continue this dynamical exploration of the properties of $\Omega(n)$. Relaxing the assumptions of Bergelson and Richter's Theorem, we obtain further results regarding the convergence of ergodic averages along $\Omega(n)$. In \cref{subsec: Proof of Pointwise}, we show that pointwise almost-everywhere convergence for $L^1$ functions does not hold in any non-atomic ergodic system: 
    
    \begin{theorem}
    \label{thm: Pointwise Convergence along Omega}
        Let $(X, \mathcal{B}, \mu, T)$ be a non-atomic ergodic dynamical system. Then there is a set $A \in \mathcal{B}$ such that for almost every $x \in X$,
        \begin{equation}
        \label{eqn: Pointwise Averages}
            \limsup_{N \to \infty} \AVG 1_A(T^{\Omega(n)}x) = 1  
            \quad 
            \text{ and } 
            \quad
            \liminf_{N \to \infty} \AVG 1_A(T^{\Omega(n)}x) = 0,
        \end{equation}
        where $1_A$ denotes the indicator function of $A$. 
    \end{theorem}
    
    In particular, \cref{thm: Pointwise Convergence along Omega} demonstrates that the assumptions in \cref{thm: Bergelson-Richter Theorem} that the system is uniquely ergodic and $g$ is continuous are not only necessary for pointwise convergence to the proper limit, but for pointwise convergence to hold at all. To prove \cref{thm: Pointwise Convergence along Omega}, the key idea is to approximate the ergodic averages along $\Omega(n)$ by weighted sums. We show that for all $\epsilon > 0$ and $N \in \N$, there are weight functions $w_*(N): \N \to \R$, supported on large intervals $I_N$, such that
    \[
        \AVG 1_A(T^{\Omega(n)}x) = \sum_{k \in I_N} w_k(N) ~ 1_A(T^k x) +  O(\epsilon),
    \]  
    as $N$ tends to infinity. Leveraging the size and placement of the intervals $I_N$, we employ a standard argument to demonstrate the failure of pointwise convergence. Moreover, our method shows that there is not just one set $A \in \mathcal{B}$ for which \eqref{eqn: Pointwise Averages} holds, but rather there exists a dense $G_\delta$ subset $\mathcal{R} \sse \mathcal{B}$ such that \eqref{eqn: Pointwise Averages} holds for every $A \in \mathcal{R}$. Thus, the operators $T^{\Omega(n)}: \mathcal{B} \to L^1(\mu)$ defined by $T^{\Omega(n)}A(x) := 1_A(T^{\Omega(n)} x)$ are shown to have the \textit{strong sweeping-out property}.  
    
    In any ergodic system, the set of generic points has full measure. Generic points are those whose ergodic averages converge to $\int_X f \, d\mu$ for every continuous function $f$ (See \cref{sec: Background} for the precise definition). In light of \cref{thm: Bergelson-Richter Theorem,thm: Pointwise Convergence along Omega}, it is natural to wonder whether convergence still holds when the ergodic averages are taken along the sequence $\Omega(n)$. However, the answer is no, and in \cref{subsec: Generic Points}, we explicitly construct a symbolic system yielding a counterexample. 
    
    Bergelson and Richter show in \cite{BR2020} that \cref{thm: Bergelson-Richter Theorem} is a direct generalization of the Prime Number Theorem. In \cref{sec: EK + Thm A}, we demonstrate the relationship of \cref{thm: Bergelson-Richter Theorem} to another fundamental result from number theory, the Erd\H{o}s-Kac Theorem. Let $C_c(\R)$ denote the set of continuous functions on $\R$ of compact support. An equivalent version of the Erd\H{o}s-Kac Theorem states that for all $F \in C_c(\R)$, 
    \[
        \lim_{N \to \infty}
        \AVG F \Big( \frac{\Omega(n) - \log \log N }{\sqrt{\log \log N}} \Big) 
        = 
        \frac{1}{\sqrt{2\pi}} \int_{-\infty}^{\infty} F(x) e^{-x^2/2} \, dx.
    \] 
    Roughly speaking, this tells us that for large $N$, the sequence $\{\Omega(n): 1 \leq n \leq N\}$ approaches a normal distribution with mean and variance $\log \log N$. We have now introduced two sequences describing distinct behaviors of \( \Omega(n) \), \( \big\{F\Big( \frac{\Omega(n) - \log \log N }{\sqrt{\log \log N}} \Big) \big\}_{n = 1}^{N} \) capturing the Erd\H{o}s-Kac Theorem and \( \{ g(T^{\Omega(n)}x)\}_{n \in \N} \) capturing \cref{thm: Bergelson-Richter Theorem}.  Two sequences $a,b: \N \to \C$ are called \textit{asymptotically uncorrelated} if 
    \[
        \AVG a(n) \overline{b(n)} = \bigg( \AVG a(n) \bigg)\bigg( \AVG b(n) \bigg) + o(1).
    \]
    In \cref{sec: EK + Thm A}, we demonstrate that \cref{thm: Bergelson-Richter Theorem} and the Erd\H{o}s-Kac Theorem exhibit a form of disjointness, in that the sequences capturing their behavior are asymptotically uncorrelated:
    
    \begin{theorem}
    \label{thm: Generalization of EK + Thm A}
        Let $(X, \mu, T)$ be uniquely ergodic and let $F \in C_c(\R)$. Then 
            \[
                \lim_{N \to \infty} \AVG F \Big( \frac{\Omega(n) - \log \log N }{\sqrt{\log \log N}} \Big) g(T^{\Omega(n) }x) 
                =
                \Big(\frac{1}{\sqrt{2\pi}} \int_{-\infty}^{\infty} F(x) e^{-x^2/2} \, dx\Big)\Big( \int_X g \, d\mu \Big)
            \]
        for all $g \in C(X)$ and $x \in X$. 
    \end{theorem}

    We show that \cref{thm: Generalization of EK + Thm A} can be viewed as a corollary of the following more general estimate. Let \( \varphi(n) = \frac{\Omega(n) - \log \log N }{\sqrt{\log \log N}}\). Then for any bounded arithmetic function \( a: \N \to \C \),
    \begin{equation}
    \label{eqn: Invariance Property}
        \AVG F(\varphi(n)) \, a(\Omega(n)) = \AVG F(\varphi(n)) \, a(\Omega(n)+1)  + o(1). 
    \end{equation}
    
    For the proof of \eqref{eqn: Invariance Property}, our strategy is to approximate each average by a double average involving dilations by primes. The key observation is that \( F(\varphi(n)) \) is asymptotically invariant under dilations by primes, whereas \( \Omega(n) \) is highly sensitive to such dilations. This sensitivity is particularly noticeable in the case that \( a(n) = (-1)^n \), so that \( a(\Omega(p n)) = - a(\Omega(n)) \). We leverage these contrasting behaviors to obtain the desired invariance in Equation \eqref{eqn: Invariance Property}. 
    
    Let $\lambda(n) = (-1)^{\Omega(n)}$ denote the classical Liouville function. Another equivalent formulation of the Prime Number Theorem states that 
    \[
        \lim_{N \to \infty} \AVG \lambda(n) = 0.
    \]  
    This formulation of the Prime Number Theorem can be seen as a special case of \cref{thm: Bergelson-Richter Theorem} by choosing $(X, \mu, T)$ to be the uniquely ergodic system given by rotation on two points (See \cite{BR2020} or \cref{sec: Background} for details). In a similar fashion, we obtain the following corollary of \cref{thm: Generalization of EK + Thm A}:
    
    \begin{corollary} 
    \label{thm: Generalization of EK + PNT}
        Let $F \in C_c(\R)$.
        Then 
        \[
            \lim_{N \to \infty} \AVG F \Big( \frac{\Omega(n) - \log\log N }{\sqrt{\log \log N}} \Big) \lambda(n) = 0.
        \]
    \end{corollary} 
    
    \cref{thm: Generalization of EK + PNT} demonstrates that the behaviors of $\Omega(n)$ captured by the Erd\H{o}s-Kac Theorem and the Prime Number Theorem exhibit disjointness. This can be interpreted as saying that, for large $N$, the sequence \( \{ \Omega(n) : 1 \leq n \leq N, \Omega(n) \text{ is even} \} \) still approaches a normal distribution with mean and variance $\log \log N$.

\section{Background Material}
\label{sec: Background}

    \subsection{Measure-preserving systems}
    \label{subsec: Dynamics}
    
    By a \textit{topological dynamical system}, we mean a pair $(X,T)$, where $X$ is a compact metric space and $T$ a homeomorphism of $X$. A Borel probability measure $\mu$ on $X$ is called \textit{T-invariant} if $\mu(T^{-1}A) = \mu(A)$ for all measurable sets $A$. By the Bogolyubov-Krylov theorem (see for instance \cite[Corollary 6.9.1]{Walters1982}), every topological dynamical system has at least one $T$-invariant measure. If a topological system $(X, T)$ admits only one such measure, $(X, T)$ is called \textit{uniquely ergodic}. 
    
    By a \textit{measure-preserving dynamical system}, we mean a probability space $(X, \mathcal{B}, \mu)$, where $X$ is a compact metric space and $\mathcal{B}$ the Borel $\sigma$-algebra on $X$, accompanied by a measure-preserving transformation $T: X \to X$. We often omit the $\sigma$-algebra $\mathcal{B}$ when there is no ambiguity. A measure-preserving dynamical system is called \textit{ergodic} if for any $A \in \mathcal{B}$ such that $T^{-1}A = A$, one has $\mu(A) = 0$ or $\mu(A) = 1$. It is easy to check that in a uniquely ergodic system, the unique invariant measure is ergodic.  
    
    One of the most fundamental results in Ergodic Theory is the Birkhoff Pointwise Ergodic Theorem, which states that for any ergodic system $(X, \mu, T)$ and $f \in L^1(\mu)$,
    \[
        \lim_{N \to \infty} \AVG f(T^n x) = \int_X f \, d\mu
    \]  
    for almost every $x \in X$. 
    
    A point $x \in X$ is called \textit{generic for the measure $\mu$} if 
    \[
        \lim_{N \to \infty} \AVG f(T^n x) = \int_X f \, d\mu
    \]
    for all $f \in C(X)$, where $C(X)$ denotes the space of continuous functions on $X$. Thus, generic points are those for which pointwise convergence holds for every continuous function. When $\mu$ is ergodic, the set of generic points has full measure.

    \subsection{Symbolic Systems}
    \label{subsec: Symbolic Dynamics}
    
    Let $\mathcal{A}$ be a finite set of symbols. Let $\mathcal{A}^\N$ denote the set of all infinite sequences with entries coming from $\mathcal{A}$. The set $\mathcal{A}^\N$ is endowed with the product topology coming from the discrete topology on the alphabet $\mathcal{A}$. In fact, this forms $\mathcal{A}^{\N}$ into a compact metric space. Denote an element in  $\mathcal{A}^\N$ by $\mathbf{x} = (x_i)_{i \in \N}$. One equivalent choice of metric generated by this topology is given by
    \[
        \text{d}(\mathbf{x}, \mathbf{y}) = 2^{-\inf \{i \in \N\,:\, x_i \neq y_i\}}. 
    \]
    This space carries a natural homomorphism $\sigma: \mathcal{A}^\N \to \mathcal{A}^\N$, called the \textit{left shift}, defined by \( (\sigma \mathbf{x})_i = x_{i+1} \).

    \subsection{Background on $\Omega(n)$}
    \label{subsec: Number Theory}
    
    Let $\Omega(n)$ denote the number of prime factors of $n$, counted with multiplicity. One equivalent formulation of the Prime Number Theorem \cite{Landau1953, Mangoldt1897} states that asymptotically, $\Omega(n)$ is even half of the time. This statement can be expressed by the classical Liouville function $\lambda(n) = (-1)^{\Omega(n)}$:
    \[
        \lim_{N \to \infty} \AVG \lambda(n) = 0.
    \]
    As a statement involving averages, this formulation is useful from a dynamical point of view. However, a rephrasing of this statement leads to several naturally stated generalizations. For a set $E \sse \N$, the \textit{natural density of E in $\N$} is defined to be
    \[
        \lim_{N \to \infty} \frac{|E \cap \{1, \dots, N\}|}{N}.
    \]
    Define $E_2 = \{n \geq 1 \,:\, \Omega(n) \equiv 0 \mod 2\}$. Then the Prime Number Theorem states that $E_2$ has natural density 1/2. In other words, $\Omega(n)$ distributes evenly over residue classes mod 2. The following theorem due to Pillai and Selberg \cite{Pillai1940, Selberg1939}  states that $\Omega(n)$ distributes over all other residue classes as well:
    
    \begin{theorem}[Pillai, Selberg]
    \label{thm: Pillai-Selberg}
        For all $m \in \N$ and $r \in \{0,\dots, m-1\}$, the set $E_m := \{n \in \N : \Omega(n) \equiv r \mod m\}$ has natural density $1/m$.
    \end{theorem}
    
    Complementing this result of Pillai and Selberg is a theorem due to Erd\H{o}s and Delange. A sequence $\{a(n)\}_{n \in \N} \sse \R$ is \textit{uniformly distributed mod 1} if 
     \[
        \lim_{N \to \infty} \AVG f(a(n)) = \int_{[0,1]} f ~ d\mu 
     \]
     for all continuous functions $f: [0,1] \to \C$. Erd\H{o}s mentions without proof \cite[p.2]{Erdos1946} and Delange later proves \cite{Delange1958} the following statement:
    
    \begin{theorem}[Erd\H{o}s, Delange]
    \label{thm: Erdos-Delange}
        Let $\alpha \in \R\setminus \Q$. Then $\{\Omega(n)\alpha\}_{n \in \N}$ is uniformly distributed mod 1. 
    \end{theorem}
    
    Bergelson and Richter's \cref{thm: Bergelson-Richter Theorem} uses dynamical methods to provide a simultaneous generalization of these number theoretic results. We review their argument for obtaining the Prime Number Theorem from \cref{thm: Bergelson-Richter Theorem} (see \cite[p.3]{BR2020} for obtaining the Pillai-Selberg and Erd\H{o}s-Delange Theorems), as we use a similar argument in \cref{subsec: Proof of EK + Thm A} to obtain \cref{thm: Generalization of EK + PNT} from \cref{thm: Generalization of EK + Thm A}. Let $X = \{0,1\}$ and define $T: X \to X$ by $T(0) = 1$ and $T(1) = 0$. Let $\mu$ be the Bernoulli measure given by $\mu(\{ 0 \}) = 1/2$ and $\mu(\{ 1 \}) = 1/2$. This system is commonly referred to as rotation on two points and is uniquely ergodic. Define a continuous function $F: X \to \R$ by $F(0) = 1$ and $F(1) = -1$. Then
    \[
        \int_X F(x) \, d\mu(x) = \frac{1}{2} F(0) + \frac{1}{2} F(1) = 0.
    \]
    Finally, one can check that
    \[
        \lambda(n) = (-1)^{\Omega(n)} = F(T^{\Omega(n)}0) .
    \]
    Hence
    \[
        \lim_{N \to \infty} \AVG \lambda(n) = \lim_{N \to \infty} \AVG F(T^{\Omega(n)}0) = 0,
    \]
    where the last equality follows by \cref{thm: Bergelson-Richter Theorem}.

    We now state two theorems that give further insight into the statistical properties of $\Omega(n)$. Hardy and Ramanujan showed that the normal order of $\Omega(n)$ is roughly $\log \log n$ \cite{HR1917}: 

    \begin{theorem}[Hardy-Ramanujan Theorem]
    \label{thm: Hardy Ramanujan}
        For $C > 0$, define $g_C: \N \to \N$ by: 
         \[
            g_C(N) = \# \Big\{ n \leq N : |\Omega(n) - \log \log n| > C \sqrt{\log \log N} \Big\}.
        \]
        Then for all $\epsilon > 0$, there is some $C \geq 1$ such that 
        \[
            \limsup_{N \to \infty} \frac{g_C(N)}{N} \leq \epsilon.
        \]
    \end{theorem}

    Erd\H{o}s and Kac later generalized this theorem to show that $\Omega(n)$ actually becomes normally distributed within such intervals \cite{EK1940}:

    \begin{theorem}[Erd\H{o}s-Kac Theorem]
    \label{thm: Erdos-Kac}
        Define $K_N: \Z \times \Z \to \N$ by:
        \[
            K_N(A,B) 
            = 
            \Big| \Big\{ n \leq N \,:\, A \leq \frac{\Omega(n) - \log\log N}{\sqrt{\log\log N}} \leq B \Big\}\Big|.
        \]
        Then
        \[
            \lim_{N \to \infty} \frac{K_N(A,B)}{N} 
            = 
            \frac{1}{\sqrt{2\pi}} \int_{A}^B e^{-t^2/2} d t.
        \]
    \end{theorem} 
    
    Thus, the Erd\H{o}s-Kac Theorem states that for large $N$, the number of prime factors of an integer $n \leq N$ becomes roughly normally distributed with mean and variance $\log \log N$. Recall that the Prime Number Theorem has an equivalent formulation in terms of averages of the Liouville function, making it well suited for dynamical settings. Similarly, the Erd\H{o}s-Kac Theorem can be stated in terms of averages. Let $C_c(\R)$ denote the set of continuous functions on $\R$ with compact support and let $F \in C_c(\R)$. \cref{thm: Erdos-Kac} is equivalent to the statement
    \[
        \lim_{N \to \infty} \AVG F \Big( \frac{\Omega(n) - \log \log N }{\sqrt{\log \log  N}} \Big) 
        = 
        \frac{1}{\sqrt{2\pi}} \int_{-\infty}^{\infty} F(x) e^{-x^2/2} \, dx.
    \]
    
    \vspace{10pt}

    One direction of this equivalence follows by setting $F(x)$ to be the indicator function on the interval $[A,B]$. The other follows from the fact that any compactly supported continuous function can be approximated by simple functions of the form 
    \[
        \sum_{k=1}^N 1_{[A_k, B_k]}(x),
    \]
    where $1_E(x)$ denotes the indicator function of the set $E$.

    \subsection{Mean Convergence}
    \label{subsec: Mean Convergence}
    
    We show that Mean Convergence holds along the sequence $\Omega(n)$.   
    
    \begin{theorem}
    \label{thm: Mean Convergence along Omega}
        Suppose $(X, \mu, T)$ is ergodic and let $f \in L^2(\mu)$. Then 
        \[
        \lim_{N \to \infty} \Big \lVert \AVG T^{\Omega(n)} f - \int_X f \, d\mu \Big \rVert_2 
        = 0.
        \]
    \end{theorem}
    
    This statement seems to be well-known, but we were not able to find a proof, so the proof is included here for completeness. 
    
    \begin{proof}[Proof of \cref{thm: Mean Convergence along Omega}]
        By a standard argument applying the Spectral Theorem, it is enough to check that for any $\beta \in (0,1)$,
        \[
            \lim_{N \to \infty} \AVG e^{2 \pi i \beta \Omega(n)} = 0.
        \]
        
        First, suppose that $\beta \in (0,1) \setminus \Q$. By \cref{thm: Erdos-Delange}, the sequence $\{\Omega(n) \beta \}$ is uniformly distributed mod 1. Then the Weyl Equidistribution Criterion (See for instance \cite{Weyl1916} or \cite[Lemma 4.17]{EW2011}) implies that 
        \[
            \lim_{N \to \infty} \AVG e^{2 \pi i \Omega(n) \beta } = 0, 
        \]
        as desired. Now, suppose that $\beta = \frac{p}{q}$, where $p,q \in \Z$ are coprime. By \cref{thm: Pillai-Selberg}, $\Omega(n)$ distributes evenly over residue classes mod $q$. It is straightforward to check this is equivalent to the statement that
        \[
            \lim_{N \to \infty} \frac{1}{N} \sum_{n=1}^N \zeta^{\Omega(n)} = 0,
        \]
        where $\zeta$ is a primitive $q$-th root of unity. Since $\gcd(p,q) = 1$, $e^{\frac{2\pi i p}{q}}$ is a primitive $q$-th root of unity, and we are done. 
    \end{proof}

\section{Counterexamples to Convergence}   
\label{sec: Counterexamples}
    The condition of unique ergodicity is essential to the proof of \cref{thm: Bergelson-Richter Theorem}. In this section, we show that, removing this assumption, convergence need not hold for an arbitrary generic point and pointwise almost-everywhere convergence does not hold in any non-atomic ergodic system.

    \subsection{Counterexample for Generic Points}
    \label{subsec: Generic Points}
    
    We show that convergence need not hold for generic points: 
    
    \begin{prop}
    \label{prop: Generic Points}
        There exists an ergodic system $(X, \mu, T)$, a generic point $x \in X$ for the measure $\mu$, and a continuous function $F \in C(X)$ such that the averages 
        \[
            \AVG F(T^{\Omega(n)} x) 
        \]
        do not converge. 
    \end{prop}
    
    We explicitly construct a symbolic system, generic point, and continuous function for which the above averages do not converge.  
    
    \begin{proof}
    	Let $(X, \sigma)$ be the one-sided shift system on the alphabet $\{0,1\}$ and let $\delta_{\mathbf{0}}$ denote the delta mass at $\mathbf{0} = (.00 ...) \in X$. Notice that $\delta_{\mathbf{0}}$ is $\sigma$-invariant and trivially ergodic. Define a sequence $\mathbf{a} \in X$ by
    	\[ 
    	    a(n) 
    	    = 
    	    \begin{cases}
        		1, & n \in [3^k - 2^k, 3^k + 2^k] \text{ for some } k \in \N
        		\\
        		0, & \text{else}
        		
    	    \end{cases}
    	    .
    	\]
    	
    	We claim that \( \mathbf{a} \) is generic for $\delta_{\mathbf{0}}$, meaning that for any $f \in C(X)$, 
    	\[
            \lim_{N \to \infty} \AVG f(\sigma^n \mathbf{a}) 
            =
            \int_X f \, d(\delta_{\mathbf{0}}) 
            = 
            f(\mathbf{0}).
        \]
        Fix $\epsilon > 0$. For $N \in \N$, define 
    	\[
    	    A_N := \{n \leq N : |f(\sigma^n \mathbf{a}) - f(\mathbf{0})| > \epsilon\}
    	\]
    	and
    	\[
    	    B_N := \{n  \leq N: |f(\sigma^n \mathbf{a}) - f(\mathbf{0})| \leq \epsilon\}. 
    	\] 
    	Then for each $N$, 
    	\begin{align*}
    	    \Big|\AVG f(\sigma^n \mathbf{a}) - f(\mathbf{0})\Big| 
    	    &\leq 
    	    \AVG |f(\sigma^n \mathbf{a})-f(\mathbf{0})|
    	    \\
    	    &= 
    	    \frac{1}{N} \sum_{n \in A_N} |f(\sigma^n \mathbf{a}) - f(\mathbf{0})| + \frac{1}{N} \sum_{n \in B_N} |f(\sigma^n \mathbf{a}) - f(\mathbf{0})|.
    	\end{align*} 
    	 
    	It is immediate from the definition of $B_N$ that 
    	\begin{equation}
    	\label{eqn: Bound for sum over BN}
    	    \frac{1}{N} \sum_{n \in B_N} |f(\sigma^n \mathbf{a}) - f(\mathbf{0})| \leq \epsilon.
    	\end{equation}
    	  
    	We now consider the sum over $A_N$. Let $M > 0$ be a bound for $|f|$. Since $f$ is continuous, there is some $\delta > 0$ such that \( d(\sigma^n \mathbf{a}, \mathbf{0}) \leq \delta \) implies $|f(\sigma^n \mathbf{a}) - f(\mathbf{0})| \leq \epsilon$. Define $C_N \sse \N$ by
    	\[
    	    C_N := \{n \leq N\,:\, d(\sigma^n \mathbf{a}, \mathbf{0}) > \delta\}.
    	\]
    	Notice that $A_N \sse C_N$ for all $N$.  Then
    	\begin{equation}
    	\label{eqn: Bound for sum over AN}
    	    \frac{1}{N} \sum_{n \in A_N} |f(\sigma^n \mathbf{a}) - f(\mathbf{0})|
    	    \leq 
    	    \frac{2M |A_N|}{N}
    	    \leq 
    	    \frac{2M|C_N|}{N}.
    	\end{equation}
    	Let $m \in \N$ be the smallest integer such that $2^{-{m+1}} \leq \delta$. Then 
    	\begin{equation}
    	\label{eqn: Bound for size of CN}
    	    |C_N| 
    	    \leq 
    	    \sum_{\{k: 3^{k-1} \leq N\}} (2^{k+1}+m) 
    	    \leq 
    	    (\log\log\log N +1)(2^{\log \log \log N+1} +m).
    	\end{equation}
    	Combining Equations \eqref{eqn: Bound for sum over AN} and \eqref{eqn: Bound for size of CN}, we obtain 
    	\[
    	    \frac{1}{N} \sum_{n \in A_N} |f(\sigma^n \mathbf{a}) - f(\mathbf{0})| \leq \epsilon
    	\]
    	for large enough $N$. Combining the estimates from Equations \eqref{eqn: Bound for sum over BN} and \eqref{eqn: Bound for sum over AN}, and letting $\epsilon \to 0$, this completes the proof of the claim. 
    	
    	Now, define $F: X \to \R$ by $F(\mathbf{x}) = \mathbf{x}(0)$, so that $a(n) = F(\sigma^n \mathbf{a})$. Since \( \mathbf{a} \) is generic for the measure $\delta_{\mathbf{0}}$, 
    	\[
    	    \lim_{N \to \infty} \AVG a(n) 
    	        = 
    	    \lim_{N \to \infty} \AVG F(\sigma^n \mathbf{a}) 
    	        = 
    	    F(\mathbf{0}) 
    	        = 
    	    0.
    	 \]
    	Define a subsequence $\{N_k\}_{k \in \N} \sse \N$ by $\log \log N_k = 3^k$. 
    	We first estimate the sum 
    	\[
    	    \frac{1}{N_k} \sum_{n=1}^{N_k} a(\log\log n)
    	\]
    	for fixed $k$. Let $I_k = [3^k-2^k, 3^k+2^k]$. It is easy to check that for $n \leq N_k$, $\log \log n$ lands in the interval $I_k$ when $n \geq N_k^{1/(2^{2^k})}$.  
    	Then 
    	\[
    	    |\{ n \leq N_k: \log \log n \in I_k\}| = N_k - \lceil N_k^{1/(2^{2^k})} \rceil.
    	\]
    	Since $a(n) = 1$ on $I_k$,
    	\[
    	    \frac{N_k - \lceil N_k^{1/(2^{2^k})}\rceil }{N_k}
    	        \leq 
    	    \frac{1}{N_k} \sum_{n=1}^{N_k}  \, a(\log\log n)
    	        \leq 
    	    1.
    	\]
    	Hence 
    	\begin{equation}
    	\label{eqn: Avgs along loglog n}
    	    \lim_{k \to \infty} \frac{1}{N_k} \sum_{n=1}^{N_k} a(\log \log n)
    	    =
    	    1.
    	\end{equation}
    	
    	We now use \cref{thm: Hardy Ramanujan} to replace \( \log \log n \) by \( \Omega(n) \) in Equation \eqref{eqn: Avgs along loglog n}. Let \( \epsilon > 0 \) and let \( C > 0 \) be that guaranteed by \cref{thm: Hardy Ramanujan}. Set 
    	\[
    	    G_{C}(N) := \{ n \leq N \,:\, |\Omega(n) - \log \log n| \geq C \sqrt{\log \log N} \}, 
    	\] 
    	so that \( |G_C(N)| = g_C(N) \). Define 
    	\[
    	    I_k' := [3^k - 2^k + C \sqrt{3^k} ~,~ 3^k + 2^k - C \sqrt{3^k}].
    	\]  
    	Then 
    	\begin{align*}
    	    \frac{1}{N_k} \sum_{n=1}^{N_k} a(\Omega(n)) 
    	    \, &\geq \, 
    	    \frac{1}{N_k} \sum_{\log \log n \in I_k'} a(\log \log n) 
    	    - 
    	    \frac{1}{N_k} \sum_{\underset{n \in G_C(N_k)}{\log \log n \in I_k'}} a(\log \log n)
    	    \\
    	    &\geq 
    	    \frac{|\{ n \leq N_k : \log \log n \in I_k'\} |}{N_k} - \frac{g_C(N)}{N_k}.
    	\end{align*}
    	One can directly calculate:
    	\[
    	    |\{ n \leq N_k : \log \log n \in I_K'\} | = N_k - \lceil N_k^{2^{(-2k + C \sqrt{3^k})}} \rceil,
    	\]
    	so that 
    	\[
    	   \limsup_{k \to \infty} \frac{1}{N_k} \sum_{n=1}^{N_k} a(\Omega(n)) 
    	   \geq 
    	   \limsup_{k \to \infty} \Big[ \frac{N_k - \lceil N_k^{2^{(-2k + C \sqrt{3^k})}} \rceil}{N_k} - \frac{g_C(N)}{N_k} \Big] 
    	   \geq 
    	   1 - \epsilon,
    	\]
    	where the last inequality follows from \cref{thm: Hardy Ramanujan}. This indicates that the averages along $\Omega(n)$ either converge to 1 or do not converge at all. However, consider the subsequence \( \{M_k\}_{k \in \N} \) given by \( \log \log M_k = 2(3^k - 2^{k-1}) \), so that \( \log \log M_k \) lands in the middle of the $k$-th interval of zero's in the definition of $\mathbf{a}$. Then by an analogous argument,  
    	\[
    	    \limsup_{k \to \infty} \frac{1}{M_k}\sum_{n=1}^{M_k} a(\Omega(n)) \leq \epsilon. 
    	\]
    	Hence the averages along $\Omega(n)$ do not converge.
    \end{proof}
    
    \cref{prop: Generic Points} tells us that convergence of Birkhoff averages is not enough to guarantee convergence of Birkhoff averages along $\Omega(n)$. However, given a stronger assumption on the convergence of the standard Birkhoff averages, convergence along $\Omega(n)$ does hold. In fact, this follows from a more general result: 
    
    \begin{prop}
        Suppose $a: \N \to \C$ is a bounded arithmetic function and $\AVG a(n)$ converges to zero uniformly. Then $\lim_{N \to \infty} \AVG a(\Omega(n)) = 0$.
    \end{prop}

    \begin{proof}
        It follows from \cite[Theorem 1.1]{Richter2021} that for any fixed $k \in \N$,
        \[
            \lim_{N \to \infty} \Big| \AVG a(\Omega(n)) - \AVG a(\Omega(n)+k) \Big| 
            = 
            0.
        \]
        Hence
        \begin{equation*}
        \label{eqn: Shifted Averages along Omega}
            \lim_{K \to \infty} \lim_{N \to \infty} \AVGK \AVG a(\Omega(n) + k) 
            =
            \lim_{N \to \infty} \AVG a(\Omega(n)),
        \end{equation*}
        assuming the limits exist. Let \( \epsilon > 0 \). Since the Ces\'aro averages of \( a \) tend to zero uniformly, there is some \( K_0 \in \N \) such that for \( K \geq K_0 \), 
        \[
            \sup_{M > 0} \Big| \frac{1}{K} \sum_{k=M}^{M+K} a(k) \Big| \leq \epsilon.
        \]
        Then for any fixed \( K \geq K_0 \), 
        \[
            \lim_{N \to \infty} \Big| \AVGK \AVG a(\Omega(n) + k) \Big| = \lim_{N \to \infty} \Big| \AVG \Big( \frac{1}{K} \sum_{k = \Omega(n)}^{\Omega(n) + K} a(k)\Big) \Big|  \leq \epsilon. 
        \]
        Hence 
        \[
            \lim_{K \to \infty} \lim_{N \to \infty} \AVG \Big( \frac{1}{K} \sum_{k = \Omega(n)}^{\Omega(n) + K} a(k)\Big) \leq \epsilon.
        \]
        for all \( \epsilon > 0 \) and we are done. 
    \end{proof}

    \begin{corollary}
        Suppose $(X, \mu, T)$ is ergodic and suppose $\AVG f(T^n x) \to 0$ uniformly. Then 
        \[
            \lim_{N \to \infty} \AVG f(T^{\Omega(n)}x) = 0.
        \]
    \end{corollary}

    \subsection{A Transition to Weighted Sums}
    \label{subsec: Weighted Sums}
        From this point on, we denote 
        \[
            \log_m N := \underbrace{\log \log \dots \log}_{m \text{ times}} N .
        \]  
        Most commonly, we take $m = 2,3$. To study pointwise convergence without the condition of unique ergodicity, we first introduce a different formulation of the ergodic averages along $\Omega(n)$. Let $a: \N \to \C$ be a bounded arithmetic function. We show that there are weight functions $w_k(N)$ so that 
        	\[
        		\AVG a(\Omega(n)) = \sum_{k \geq 0} w_k(N) a(k).
        	\]
        Regrouping the terms by the value of $\Omega(n)$ yields an exact formulation for these weights. Let $\pi_k(N)$ denote the number of integers not exceeding $N$ with exactly $k$ prime factors, counted with multiplicity. Then
        	\[
        		\AVG a(\Omega(n)) = \sum_{k \geq 0} \frac{\pi_k(N)}{N} ~ a(k)
        	\]	
        so that $w_k(N) = \pi_k(N)/N$. However, this exact formulation does not give much insight into the shape of these weight functions. Instead, we rely on an estimate of Erd\H{o}s to show that on large intervals, the weight functions \( w_k (N) \) can be approximated (uniformly in $k$) by a Gaussian with mean and variance \(\log_2 N \):
        
        \begin{lemma}
        \label{lem: Erdos to Gaussian}
            \[
                \frac{\pi_{k}(N)}{N} 
                = 
                \frac{1}{\sqrt{2\pi \log_2 N}} e^{-\frac{1}{2} \big(\frac{k - \log_2 N}{\sqrt{\log_2 N}} \big)^2} (1+ o(1)).
            \]
        \end{lemma}
        
        \begin{remark}
            This can be shown directly in terms of probability theory. Erd\H{o}s's estimate can be viewed as saying that $\pi_m(N)/N$ is approximated by a Poisson distribution with parameter $\log_2 N$. Since this parameter tends to infinity with $N$, for large values of $N$, we can approximate this Poisson distribution by a Gaussian distribution with mean and variance $\log_2 N$. However, since we do not take a probabilistic viewpoint in this paper, the computation is included for completeness. 
        \end{remark} 
        
        \begin{proof}[Proof of \cref{lem: Erdos to Gaussian}]
            Let \( C > 0\) be that given by \cref{thm: Hardy Ramanujan} and set 
            \[
                I_N = [\log_2 N - C \sqrt{\log_2 N} ~,~ \log_2 N + C \sqrt{\log_2 N}].
            \]
            By a result of Erd\H{o}s \cite[Theorem II]{Erdos1948},
            \[
                \frac{\pi_k(N)}{N} = \frac{1}{\log N} \frac{(\log_2 N)^{k-1}}{(k-1)!} + o(1),
            \]
            and this estimate is uniform for \( k \in I_N \). Applying Stirling's formula, 
            \begin{equation}
            \label{eqn: Stirling}
                \frac{1}{\log N} \frac{(\log_2 N)^{k-1}}{(k-1)!}
                = 
                \frac{1}{\log N} \frac{(\log_2 N)^{k-1} e^{k-1}}{(k-1)^{k-1} \sqrt{2 \pi (k-1)}} ~ (1 + o(1)).
            \end{equation}
            We now rewrite $k$ in the following form:
            \[
                k = \log_2 N + A \sqrt{\log_2 N},
            \]
            for some $A \in \R$. Taking $\log$ of Equation \eqref{eqn: Stirling} and evaluating at such values of $k$, 
                
            \begin{align*}
                \log \Big( \frac{1}{\log N} \cdot \frac{(\log_2 N)^{k-1}}{(k-1)!} \Big) 
                &=
                - \Big(\log_2 N + A \sqrt{\log_2 N} - \frac{1}{2}\Big) \log\Big( 1+ \frac{A \sqrt{\log_2 N} - 1}{\log_2 N} \Big) 
                \\
                & \quad \quad \quad \quad 
                - \frac{1}{2} \log_3 N + A \sqrt{\log_2 N} - 1 - \frac{1}{2} \log(2 \pi) + o(1).
            \end{align*}
            Next, using the quadratic approximation $\log(1+\epsilon) = \epsilon - \epsilon^2/2 + O(\epsilon^3)$ for $\epsilon < 1$, we obtain:  
            \begin{align*}
                \log \Big( \frac{1}{\log N} \cdot &\frac{(\log_2 N)^{k-1}}{(k-1)!} \Big) 
                \\
                &=
                - \Big(\log_2 N + A \sqrt{\log_2 N} - \frac{1}{2}\Big) \Big( \frac{A \sqrt{\log_2 N} - 1}{\log_2 N} - \frac{(A \sqrt{\log_2 N} - 1)^2}{2 (\log_2 N)^2 } \Big) 
                \\
                & \quad \quad \quad \quad 
                - \frac{1}{2} \log_3 N + A \sqrt{\log_2 N} - 1 - \frac{1}{2} \log(2 \pi) + o(1).
                \\
                &= 
                - A^2 - \frac{1}{2} \log_3 N - \frac{1}{2} \log(2\pi) + o(1).
            \end{align*}
            Exponentiating, we obtain the following estimate for $\pi_k(N)/N$:
                \[
                    \frac{\pi_{k}(N)}{N} 
                    = 
                    \frac{1}{\log N} \frac{(\log_2 N)^{k-1}}{(k-1)!} (1 + o(1)) 
                    = 
                    \frac{1}{\sqrt{2\pi}} \frac{1}{\sqrt{\log_2 N}} e^{- \frac{1}{2} A^2} (1+ o(1)).
                \]
            Rewriting $A$ in terms of $k$, 
            \begin{equation*}
            \label{eqn: Approx by Guassian}
                \frac{\pi_{k}(N)}{N} 
                = 
                \frac{1}{\sqrt{2\pi}} \frac{1}{\sqrt{\log_2 N}} e^{-\frac{1}{2} \big(\frac{k - \log_2 N}{\sqrt{\log_2 N}} \big)^2} (1+ o(1)).
            \end{equation*}
        \end{proof}

    \subsection{Failure of Pointwise Convergence}
    \label{subsec: Proof of Pointwise}
        We conclude this section by demonstrating the failure of pointwise convergence along $\Omega(n)$ in every non-atomic, ergodic system. The strategy is to first approximate the ergodic averages using \cref{lem: Erdos to Gaussian}. We then use the Rokhlin Lemma to construct a set of small measure on which the ergodic averages along $\Omega(n)$ become large. A lemma from functional analysis then implies the failure of pointwise convergence. This lemma is proven in much greater generality in \cite[Theorem 5.4]{RosenblattWierdl1994}, but here we state it only for the averaging operators $T_N: \mathcal{B} \to L^1(\mu)$ defined by:
        \begin{equation}
        \label{eq: Def of Avg Operators}
            T_N A(x) 
            := 
            \frac{1}{\sqrt{2\pi \log_2 N}}  \sum_{k=\lceil \log_2 N - C \sqrt{\log_2 N} \rceil }^{\lfloor \log_2 N + C\sqrt{\log_2 N}\rfloor } e^{- \frac{1}{2} \big(\frac{k - \log_2 N}{\sqrt{\log_2 N}}\big)^2} 1_A(T^k x), 
        \end{equation}
        where $C > 0$ is a constant to be chosen later. 
        
        \begin{lemma}
        \label{lem: Banach}
            Let $T_N$ be defined as in \eqref{eq: Def of Avg Operators} and let $N_0 \in \N$. Assume that for all $\epsilon > 0$ and $N \geq N_0$, there is a set $A \in \mathcal{B}$ with $\mu(A) < \epsilon$ and 
            \[
            \mu \Big\{ x \in X ~:~ \sup_{n \geq N} T_n A(x) > 1 - \epsilon \Big\} \geq 1 - \epsilon.
            \]
            Then there is a dense $G_\delta$ subset $\mathcal{R} \subset \mathcal{B}$ such that for $A \in R$, 
            \[
            \limsup_{n \to \infty} T_n A(x) = 1 \text{ for a.e. } x \in X
            \]
            and
            \[
            \liminf_{n \to \infty} T_n A(x) = 0 \text{ for a.e. } x \in X.
            \]
        \end{lemma}
        
        Operators satisfying the conclusion of \cref{lem: Banach} are said to have the \textit{strong sweeping-out property}. As we demonstrate in the case of the operators $T_N$, averaging operators with strong sweeping-out property fail for pointwise convergence. 
    
        \begin{proof}[Proof of \cref{thm: Pointwise Convergence along Omega}]
            We want to find a set $A \in \mathcal{B}$ such that for almost every $x \in X$, the averaging operators
            \[
                T_N'(A) := \frac{1}{N} \sum_{n=1}^N 1_A(T^{\Omega(n)}x)
            \]
            satisfy \eqref{eqn: Pointwise Averages}. For each $N \in \N$, define
            \[
                I_N = [\log_2 N - C\sqrt{\log_2 N} ~,~ \log_2 N + C\sqrt{\log_2 N}],
            \]  
            where the constant $C$ is chosen later. Then
            \begin{align*}
                T_N' A(x)  
                &= 
                \frac{1}{N} \sum_{k \geq 0} \pi_k(N) 1_A(T^k x) 
                \\
                &= 
                \frac{1}{N} \sum_{k \in I_N} \pi_k(N) 1_A(T^k x) + \frac{1}{N} \sum_{k \notin I_N} \pi_k(N) 1_A(T^k x).
            \end{align*}
            
            Let $\epsilon \in (0,1)$. Choose $C > 0$ satisfying \cref{thm: Hardy Ramanujan}. Then
            \begin{equation}
            \label{eqn: Estimate for not in I}
                \liminf_{N \to \infty} \Big| \frac{1}{N} \sum_{k \notin I_N} 1_A(T^k x) \pi_k(N) \Big|
                \leq 
                \liminf_{N \to \infty} \frac{1}{N} \sum_{k \notin I_N} \pi_k(N) 
                =
                \liminf_{N \to \infty} \frac{g_C(N)}{N} 
                \leq 
                \epsilon.
            \end{equation}
                
            For $k \in I_N$, we approximate $\pi_k(n)/N$ using \cref{lem: Erdos to Gaussian}. Since this estimate is uniform over $k \in I_N$, there are \(\epsilon_N: \N \to \R\) such that \( \lim_{N \to \infty} \sup_{k \in I_N} |\epsilon_N(k)| = 0\) and  
            \[
                \sum_{k \in I_N} \frac{\pi_k(N)}{N} 1_A(T^k x) 
                = 
                \sum_{k \in I_N} \Big[ \frac{1}{\sqrt{2\pi \log_2 N}}  e^{- \frac{1}{2}(\frac{k- \log_2 N}{\sqrt{\log_2 N}})^2} ( 1 + \epsilon_N(k) )\Big] 1_A(T^k x).
            \]
            Since $|\epsilon_N(k)|$ tends to zero uniformly in $k$ as $N$ tends to infinity,  
            \begin{equation}
            \label{eqn: Estimate for in I}
                \sum_{k \in I_N} \frac{\pi_k(N)}{N} 1_A(T^k x)
                = \frac{1}{\sqrt{2\pi \log_2 N}} \sum_{k = \lceil \log_2 N - C \sqrt{\log_2 N}\rceil }^{\lfloor \log_2 N + C \sqrt{\log_2 N} \rfloor} e^{- \frac{1}{2}(\frac{k- \log_2 N}{\sqrt{\log_2 N}})^2} 1_A(T^k x) + o(1).
            \end{equation}
            From \cref{eqn: Estimate for not in I,eqn: Estimate for in I}, we obtain
            \begin{equation}
            \label{eqn: limsup estimates for Tn}
                \limsup_{N \to \infty} T_N A (x) 
                \leq  
                \limsup_{N \to \infty} T_N' A (x) 
                \leq 
                1,
            \end{equation}
            and
            \begin{equation}
            \label{eqn: liminf estimates for Tn}
                0 
                \leq 
                \liminf_{N \to \infty} T_N' A (x)
                = 
                \liminf_{N \to \infty} T_N A (x) + \epsilon.
            \end{equation}
            
            The upper bound from \cref{eqn: limsup estimates for Tn} and the lower bound from \cref{eqn: liminf estimates for Tn} are trivial. Now, fix \(N_0 \in \N \) such that for all \(M \geq N_0\),
            \[
                \Bigg| \Bigg( \frac{1}{\sqrt{2\pi \log_2 M}}  \sum_{k=\lceil \log_2 M - C \sqrt{\log_2 M}\rceil }^{\lfloor \log_2 M + C \sqrt{\log_2 M}\rfloor } e^{-\big(\frac{k - \log_2 M}{\sqrt{\log_2 M}}\big)^2} \Bigg)
                - 
                1 \Bigg| 
                < 
                \epsilon .
            \] 
            Let $N \geq M \geq N_0$. By the Rokhlin Lemma, there is a set $E$ such that the sets $T^k E$ are pairwise disjoint for $k = 0, \dots, \lfloor \log_2 N + C \sqrt{\log_2 N} \rfloor $ and 
            \[
                1 - \epsilon \leq \mu \Bigg( \bigcup_{k=0}^{\lfloor\log_2 N + C \sqrt{\log_2 N}\rfloor } T^k E \Bigg) \leq 1.
            \]
            Note the upper bound is trivial. Using the disjointness condition, we obtain bounds for the measure of $E$: 
            \begin{equation}
            \label{eqn: Measure of E estimates}
                \frac{1 - \epsilon}{\lfloor \log_2 N + C\sqrt{\log_2 N}\rfloor } 
                \leq 
                \mu(E) 
                \leq  \frac{1}{\lfloor \log_2 N + C\sqrt{\log_2 N} \rfloor}.
            \end{equation}
            
            Define $A_N := \bigcup_{k = \lceil \log_2 N - C\sqrt{\log_2 N}\rceil }^{\lfloor \log_2 N + C\sqrt{\log_2 N}\rfloor } T^k E$. Using the upper bound given in Equation \eqref{eqn: Measure of E estimates},
            \[
                \mu(A_N) 
                    \leq  
                2 \sqrt{\log_2 N } \, \mu(E) 
                    \leq
                \frac{2 \sqrt{\log_2 N}}{\lfloor \log_2 N + C\sqrt{\log_2 N} \rfloor}.
            \]
            Hence the measure of $A_N$ tends to zero as $N$ tends to infinity, so that, for large $N$, $A_N$ satisfies the first condition of \cref{lem: Banach}. Now, we show that there are sets $B_N$, $N \in \N$, such that $\mu(B_N) \to 1 - \epsilon$ as $N \to \infty$ and 
            \[
                \mu\Big\{ \sup_{M \leq k \leq N} T_k A(x) \geq 1 - \epsilon \Big\} \geq \mu(B_N).
            \]
            Let $j \in \{0 , \dots, N - M\}$. Then $M \leq N - j \leq N$. Define
            \[
                \kappa(j) := \lfloor \log_2 N - C \sqrt{\log_2 N} \rfloor -  \lfloor \log_2 (N-j) - C \sqrt{\log_2 (N-j)} \rfloor \in \Z.
            \]
            For $x \in T^{\kappa(j)}E $, we have $T^k x \in A$ for $k = \lceil \log_2 (N-j) - C \sqrt{\log_2 (N-j)} \rceil , \dots,$ $\lfloor \log_2 (N-j) + C \sqrt{\log_2 (N-j)} \rfloor $ so that
            \begin{align*}
                T_{N-j}A(x) 
                    &= 
                \frac{1}{\sqrt{2\pi \log_2 (N-j)}}  \sum_{k=\lceil \log_2 (N-j) - C \sqrt{\log_2 (N-j)}\rceil }^{\lfloor \log_2 (N-j) + C \sqrt{\log_2 (N-j)}\rfloor } e^{-\frac{1}{2} \big(\frac{k - \log_2 (N-j)}{\sqrt{\log_2 (N-j)}}\big)^2}1_A(T^k x) \\
                    &= 
                \frac{1}{\sqrt{2\pi \log_2 (N-j)}}  \sum_{k= \lceil \log_2 (N-j) - C \sqrt{\log_2 (N-j)}\rceil }^{ \lfloor \log_2 (N-j) + C \sqrt{\log_2 (N-j)} \rfloor } e^{-\frac{1}{2} \big(\frac{k - \log_2 (N-j)}{\sqrt{\log_2 (N-j)}}\big)^2}\\
                    & \geq
                1 - \epsilon,
            \end{align*}
            where the last inequality holds since $N-j \geq M \geq N_0$. Then for all $j \in \{0, \dots, N - M\}$, each $x \in T^{\kappa(j)}E$ is such that 
            \[
                \sup_{M \leq k \leq N} T_k A(x) \geq 1 - \epsilon. 
            \]
            Define $B_N := \bigcup_{j = 0}^{N-M} T^{\kappa(j)} E$. 
            Since $0 \leq \kappa(j) \leq \kappa(N-M)$ are all integer valued, $B_N$ is the union of only $\kappa(N-M)$ many disjoint sets. Then, using the lower bound from Equation \eqref{eqn: Measure of E estimates},
            \begin{align*}
                \mu\Big\{ \sup_{M \leq k \leq N} T_k A(x) \geq 1 - \epsilon \Big\}
                    \geq 
                \mu(B_N) 
                    \geq
                \frac{\kappa(N-M) (1- \epsilon)}{\lfloor \log_2 N + C \sqrt{\log_2 N}\rfloor }.
            \end{align*}
            Hence letting $N \to \infty$,
            \[
                \mu\Big\{ \sup_{M \leq k} T_k A(x) \geq 1 - \epsilon \Big\} \geq 1 - \epsilon.
            \]
        Now, take $N \geq N_0$ large such that $\mu(A_N) \leq \epsilon$ and set $A = A_N$. Then $A$ satisfies the hypothesis of \cref{lem: Banach}, and we obtain a $G_\delta$ subset $\mathcal{R} \subset \mathcal{B}$ such that for $A \in \mathcal{R}$ and a.e. $x \in X$, 
        \[
            \limsup_{N \to \infty} T_N A(x) = 1
            \qquad \text{ and } \quad
            \liminf_{N \to \infty} T_N A(x) = 0.
        \]
        \cref{eqn: limsup estimates for Tn,eqn: liminf estimates for Tn} then yield
        \[
            \limsup_{N \to \infty } T_N ' A (x) = 1 \text{ for a.e. } x \in X
        \]
        and 
        \[
            \liminf_{N \to \infty } T_N ' A (x) \leq \epsilon \text{ for a.e. } x \in X,
        \]
        for all $\epsilon > 0$.  
        \end{proof}

\section{Independence of the Erd\H{o}s-Kac Theorem and \cref{thm: Bergelson-Richter Theorem}}
\label{sec: EK + Thm A}
        
    \subsection{A logarithmic version of Prime Number Theorem}
    \label{subsec: Weak PNT}
        Let \( B \subset \N \) be a finite, non-empty subset of the integers. For a function \( f: B \to \C \), we define the \textit{Ces\'aro averages of f over B} by
        \[
            \ECC{n}{B} f(n) := \frac{1}{|B|} \sum_{n \in B} f(n)
        \]
        and the \textit{logarithmic averages of f over B} by
        \[
            \Elogg{n}{B} f(n) := \frac{1}{\sum_{n \in B} 1/n} ~ \sum_{n \in B} f(n)/n.
        \]  
        Define $[N] := \{1, 2, \dots, N\}$ and let $\P$ denote the set of primes. For $k \in \N$, let $\P_k$ denote the set of \textit{k-almost primes}, the integers with exactly $k$ prime factors, not necessarily distinct. Our averaging set $B$ is often chosen from the aforementioned sets.
        
        Before moving on to \cref{thm: Generalization of EK + Thm A}, we first present a proof of a logarithmic version of the Prime Number Theorem, as it illustrates a streamlined version of the core ideas that arise in the proof of \cref{thm: Generalization of EK + Thm A}.
        
        \begin{theorem}[Logarithmic Prime Number Theorem]
        \label{thm: Weak PNT}
            Let $\lambda(n)$ denote the Liouville function. Then 
            \[
                \lim_{N \to \infty} \Elog \lambda(n) = 0.
            \]
        \end{theorem}
    
         One of the main tools that allows us to simplify the argument in the case of logarithmic averages is the following standard trick:  

        \begin{lemma}
        \label{lem: Logarithmic Averages Trick}
            Let $f: \N \to \C$ be a bounded arithmetic function. Then for any $p \in \N$, 
            \[
                \lim_{N \to \infty} \Big| \Elogg{n}{[N/p]} f(n) - \Elog f(n) \Big| = 0. 
            \]
        \end{lemma}
        
        Intuitively, this is due to the weight of $1/n$ that logarithmic averages place on each term. As $N$ becomes large, the terms between $N/p$ and $N$ are weighted so heavily that they contribute very little to the overall average.
        
        \begin{proof}
            Let $M > 0$ be a bound for $|f|$. For $N \in \N$, define $A_N := \sum_{n=1}^N 1/n$. We calculate:
            \begin{align*}
                \Big| \Elog f(n) - \Elogg{n}{[N/p]} f(n) \Big| 
                &\leq 
                \Big( \frac{1}{A_{\lfloor N/p \rfloor}} - \frac{1}{A_N} \Big) \sum_{n=1}^{\lfloor N/p \rfloor} \frac{|f(n)|}{n} + \frac{1}{A_N} \sum_{n=\lfloor N/p \rfloor}^{N} \frac{|f(n)|}{n} 
                \\
                &\leq 
                M \Big[ \Big( 1- \frac{A_{\lfloor N/p \rfloor}}{A_N} \Big) + \frac{1}{A_N} \sum_{n=\lfloor N/p \rfloor }^{N} \frac{1}{n} \Big]
                \\
                &= 
                2 M \Big( 1- \frac{A_{\lfloor N/p \rfloor}}{A_N} \Big).
            \end{align*}
            We now use the fact that
            \[
                \lim_{N \to \infty} \Big| \log N - A_N \Big|
                = 
                \gamma,
            \]
            where \( \gamma \) denotes the Euler-Mascheroni constant. Then
            \begin{align*}
                 \lim_{N \to \infty} \frac{A_{\lfloor N/p \rfloor }}{A_N} 
                 = 
                 \lim_{N \to \infty} \frac{\log \lfloor N/p \rfloor }{\log N} 
                 = 
                 \lim_{N \to \infty} \Big( 1 - \frac{\log p}{\log N}\Big) 
                 = 
                 1,
            \end{align*}
            so that \( \lim_{N \to \infty} \Big| \Elog f(n) - \Elogg{n}{[N/p]} f(n) \Big| = 0 \).
        \end{proof}
        
        To prove \cref{thm: Weak PNT}, we also need the following proposition. The proof of statement (1) uses only elementary methods and can be found in \cite[Proposition 2.1]{BR2020}. The proof of statement (2) is completely analogous, replacing Ces\'aro averages by logarithmic averages.  
        
        \begin{prop}
        \label{prop: Cor of Turan-Kubilius}
            Let $B \sse \N$ be finite and non-empty. Define $\Phi(n,m) = \gcd (n,m) -1$ and let $1_{m|n}$ take value 1 if $m$ divides $n$ and zero else. Then
            \begin{enumerate}
                \item $\limsup_{N \to \infty} \EC \bigg| \Elogg{m}{B} (1- m 1_{m|n}) \bigg| \leq \bigg( \Elogg{m}{B} \Elogg{n}{B} \Phi(n,m) \bigg)^{1/2}$;
                
                \item $\limsup_{N \to \infty} \Elog \bigg| \Elogg{m}{B} (1- m 1_{m|n}) \bigg| \leq \bigg( \Elogg{m}{B} \Elogg{n}{B} \Phi(n,m) \bigg)^{1/2}$.
            \end{enumerate}
        \end{prop}
        
        In the following, we choose the set $B$ so that the quantity \( \Elogg{m}{B} \Elogg{n}{B} \Phi(n,m) \) is arbitrarily small. Intuitively, this means that two random elements from $B$ have a high chance of being coprime.

        \begin{proof}[Proof of Theorem \ref{thm: Weak PNT}]
            Let $\epsilon > 0$. By definition,
            \begin{align*}
                \Elogg{m}{\P \cap [s]} \Elogg{n}{\P \cap [s]} \Phi(m,n) 
                &= 
                \frac{1}{ \big( \sum_{m \in \P \cap [s]} 1 / m \big) ^2 } \sum_{m,n \in \P \cap [s]} \frac{\Phi(m,n)}{m n} \quad .
            \end{align*}
            Notice that for $m$ and $n$ from $\P \cap [s]$, 
            \[
                \Phi(m,n) = 
                \begin{cases}
                    m-1 & m=n \\
                    0   & m \neq n
                \end{cases}
                \quad .
            \]
            Then 
            \begin{align*}
                \frac{1}{ \big( \sum_{m \in \P \cap [s]} 1 / m \big) ^2 } \sum_{m,n \in \P \cap [s]} \frac{\Phi(m,n)}{m n} 
                &= 
                \frac{1}{ \big( \sum_{m \in \P \cap [s]} 1 / m \big) ^2 } \sum_{m \in \P \cap [s]} \frac{m-1}{m^2}
                \\
                &\leq
                \frac{1}{ \sum_{m \in \P \cap [s]} 1 / m }
                \quad , 
            \end{align*}
            so that
            \[
                \limsup_{s \to \infty} \Elogg{m}{\P \cap [s]} \Elogg{n}{\P \cap [s]} \Phi(m,n) = 0. 
            \]
            Take $s_0 \in \N$ such that for all $s \geq s_0$, \[
            \Elogg{m}{\P \cap [s]} \Elogg{n}{\P \cap [s]} \Phi(m,n) \leq \epsilon^2.
            \]
            Fix $s \geq s_0$. By \cref{prop: Cor of Turan-Kubilius}, 
            \[
                \limsup_{N \to \infty} \Elog \Big| \Elogg{s}{\P \cap [s]} (1-1_{p|n}) \Big| < \epsilon.
            \]
            Then by a direct calculation,
            \begin{align*}
                \Big| \Elog \lambda(n) - \Elogg{p}{\P \cap [s]} \Elogg{n}{[N/p]}  \lambda(p n) \Big| 
                &= 
                \Big| \Elog \lambda(n) - \Elogg{p}{\P \cap [s]} \Elog p 1_{p|n} \lambda(n) \Big| + O(1/\log N)
                \\
                &\leq 
                \Elog \Big| \Elogg{p}{\P \cap [s]}(1-1_{p|n})  \Big| + O(1/\log N)
                \\
                &\leq 
                \epsilon + O(1/\log N).
            \end{align*} 
            Hence 
            \[
                \lim_{N \to \infty} \Big| \Elog \lambda(n) - \Elogg{p}{\P \cap [s]} \Elogg{n}{[N/p]}  \lambda(p n) \Big| \leq \epsilon.
            \]
            Since $\lambda(p n) = - \lambda(n)$ for any prime $p$, this reduces to
            \[
                \lim_{N \to \infty} \Big| \Elog \lambda(n) + \Elogg{p}{\P \cap [s]} \Elogg{n}{[N/p]}  \lambda(n) \Big| \leq \epsilon.
            \]
            Finally, applying \cref{lem: Logarithmic Averages Trick} to $\lambda(n)$, we remove the dependence of the inner logarithmic average on $p$: 
            \[
                \lim_{N \to \infty} \Big| \Elog \lambda(n) + \Elogg{p}{\P \cap [s]} \Elogg{n}{[N]}  \lambda(n) \Big| \leq \epsilon.
            \]
            Letting \( \epsilon \to 0 \), we conclude that \(\lim_{N \to \infty} \Elog \lambda(n) = 0\).
        \end{proof}

    \subsection{Proof of \cref{thm: Generalization of EK + Thm A}}
    \label{subsec: Proof of EK + Thm A}
    
        A more technical version of this argument can be applied to obtain \cref{thm: Generalization of EK + Thm A}. The main difficulty arises in the last two steps, in which we remove the dependence of the inner average on the primes $p$. To get around this, we rely on the following technical proposition:
    
        \begin{prop}
        \label{prop: Two sets Lemma}
            For all \( \epsilon \in (0,1) \) and \( \rho \in (1, 1+\epsilon] \), there exist finite, non-empty sets \( B_1 \), \( B_2 \sse \N\)  with the following properties:
            \begin{enumerate}[i.]
                \item \( B_1 \subset \P \)  and \( B_2 \subset \P_2 \);
                
                \item \( B_1 \) and \( B_2 \) have the same cardinality when restricted to \( \rho \)-adic intervals: \newline \( |B_1 \cap (\rho^j, \rho^{j+1}]| = |B_2 \cap (\rho^j, \rho^{j+1}]| \) for all \( j \in \N \cap \{0\} \);
                
                \item \( \Elogg{m}{B_i} \Elogg{n}{B_i} \Phi(m,n) \leq \epsilon \) for \( i = 1,2\);
                
                \item for any $a: \N \to \C$ with $|a| \leq M$ for some $M > 0$,
                \[
                    \bigg| \Elogg{p}{B_1} \ECC{n}{[N/p]} a(n) - \Elogg{p}{B_2} \ECC{n}{[N/p]} a(n) \bigg| \leq 3M \epsilon. 
                \]
            \end{enumerate}
        \end{prop}
        
        \begin{proof}
            Statements \( (i) - (iii) \) can be found in \cite[Lemma 2.2]{BR2020}. Statement \( (iv) \) follows from \( (iii) \). The proof for arithmetic functions of modulus 1 can be found in \cite[Lemma 2.3]{BR2020}. The argument for a general bounded arithmetic function is completely analogous.
        \end{proof}

        \begin{prop}
        \label{prop: Invariance Property}
            Define \( \varphi(n) := \frac{\Omega(n) - \log \log N }{\sqrt{\log \log N}}\). Then for any bounded arithmetic function \( a: \N \to \C \),
            \begin{equation}
                \AVG F(\varphi(n)) \, a(\Omega(n)) = \AVG F(\varphi(n)) \, a(\Omega(n)+1)  + o(1). 
            \end{equation}
        \end{prop}
        
        \begin{proof}
            Let $\epsilon \in (0,1)$ and $\rho \in [1, 1+ \epsilon)$. Let $B_1$ and $B_2$ be finite, non-empty sets satisfying the conditions of \cref{prop: Two sets Lemma}. Let $M_1$ be a bound for $|F|$ and $M_2$ a bound for $|a|$. Then
            \begin{align*}
                \Big| \EC F(\varphi(n)) \, a(\Omega(n)+1) - & \Elogg{p}{B_1} \ECC{n}{[N/p]} F(\varphi(p n)) \, a(\Omega(p n)+1)  \Big| 
                \\
                &\quad \quad \leq 
                \EC M_1 M_2 | \Elogg{p}{B_1} (1 - p 1_{p|n})| + O(1/N) 
                \\
                &\quad \quad \leq 
                M_1 M_2 \epsilon + O(1/N), 
            \end{align*}
            where the last inequality follows by \cref{prop: Cor of Turan-Kubilius}. Hence 
            \begin{equation}
            \label{eqn: 7}
                \EC F(\varphi(n)) \, a(\Omega(n)+1) 
                = 
                \Elogg{p}{B_1} \ECC{n}{[N/p]} F(\varphi(p n)) \, a(\Omega(p n)+1)  + O(\epsilon + 1/N).
            \end{equation}
            
            Replacing $B_1$ by $B_2$ in the above argument, we obtain
            \begin{equation}
            \label{eqn: 8} 
                \EC F(\varphi(n)) \, a(\Omega(n)) 
                = 
                \Elogg{p}{B_2} \ECC{n}{[N/p]} F(\varphi(p n)) \, a(\Omega(p n))  + O( \epsilon + 1/N).
            \end{equation}
            
            Since $B_1$ consists only of primes and $B_2$ consists only of 2-almost primes, Equations \eqref{eqn: 7} and \eqref{eqn: 8} yield
            \begin{equation}
            \label{eqn: 9}
                \EC F(\varphi(n)) \, a(\Omega(n)+1) 
                = 
                \Elogg{p}{B_1} \ECC{n}{[N/p]} F(\varphi(p n)) a(\Omega(n)+2)  + O(\epsilon + 1/N)
            \end{equation}
            and
            \begin{equation}
            \label{eqn: 10}
                \EC F(\varphi(n)) \, a(\Omega(n)) 
                = 
                \Elogg{p}{B_2} \ECC{n}{[N/p]} F(\varphi(p n))\, a(\Omega(n)+2)  + O( \epsilon + 1/N),
            \end{equation}
            respectively. Since $\lim_{N \to \infty} |F(\varphi(p n)) - F(\varphi(n))| = 0$ for any fixed $p \in \N$, let $N_0$ be such that for $N \geq N_0$, $|F(\varphi(p n)) - F(\varphi(n))| \leq \epsilon$ for all $p \in B_1, B_2$. Then for $i = 1,2$,
            \begin{align*}
                |\Elogg{p}{B_i} \ECC{n}{[N/p]} F(\varphi(p n)) & a(\Omega(n)+2) - \Elogg{p}{B_i} \ECC{n}{[N/p]} F(\varphi(n)) a(\Omega(n)+2)|
                \\
                &\leq 
                \Elogg{p}{B_i}\ECC{n}{[N/p]}M_2 |F(\varphi(p n)) - F(\varphi(n))|
                \\
                &\leq 
                \frac{p M_2}{N} \Elogg{p}{B_i} \sum_{n=1}^{N_0} |F(\varphi(p n)) - F(\varphi(n))| + \Elogg{p}{B_i} p M_2 \epsilon
                \\
                &\leq 
                \frac{C_1}{N} + C_2 \epsilon,  
            \end{align*}
            where $C_1$ and $C_2$ do not depend on $N$ or $\epsilon$. Hence we can remove the dependence on $p$ from the summands of Equations \eqref{eqn: 9} and \eqref{eqn: 10}, yielding
            \begin{equation}
            \label{eqn: 11}
                \EC F(\varphi(n)) \, a(\Omega(n)+1) = \Elogg{p}{B_1} \ECC{n}{[N/p]} F(\varphi(n)) a(\Omega(n)+2)  + O(\epsilon + 1/N) 
            \end{equation}
            and
            \begin{equation}
            \label{eqn: 12}
                \EC F(\varphi(n)) \, a(\Omega(n)) = \Elogg{p}{B_2} \ECC{n}{[N/p]} F(\varphi(n))\, a(\Omega(n)+2)  + O( \epsilon + 1/N).
            \end{equation}
            Finally, \cref{prop: Two sets Lemma} yields 
            \begin{equation}
            \label{eqn: 13}
                \Elogg{p}{B_1} \ECC{n}{[N/p]} F(\varphi(n)) a(\Omega(n)+2) 
                = 
                \Elogg{p}{B_2} \ECC{n}{[N/p]} F(\varphi(n))\, a(\Omega(n)+2) + O(\epsilon).
            \end{equation}
                
            Combining Equations \eqref{eqn: 11}, \eqref{eqn: 12}, and \eqref{eqn: 13} and letting $N \to \infty$, $\epsilon \to 0$, we are done.
        \end{proof}

        \begin{proof}[Proof of \cref{thm: Generalization of EK + Thm A}]
            Fix $F \in C_c(\R)$ and $x \in X$. We first perform a reduction using the condition of unique ergodicity. Set $\varphi(n) = \frac{\Omega(n) - \log\log N}{\sqrt{\log\log N}}$. For $N \in \N$, define the measure $\mu_N$ by:
            \[
                \AVG (F \circ \varphi)(n) \, g(T^{\Omega(n)}x) = \int_X g \, d\mu_N. 
            \]
            Explicitly, $\mu_N = \frac{1}{N} \sum_{n=1}^N (F \circ \varphi)(n) \, \delta_{T^{\Omega(n)}x}$, where $\delta_y$ denotes the point mass at $y$. Now, define 
            \[
                \mu' := \Big( \frac{1}{\sqrt{2\pi}} \int_{-\infty}^\infty F(t) e^{-t^2/2} \, dt \Big) \cdot \mu. 
            \]
            Then the conclusion of the theorem is equivalent to convergence of the sequence  $\{ \mu_N \}_{n \in \N}$ to $\mu'$ in the weak-* topology. Notice that if each limit point of $\{\mu_N\}_{n \in \N}$ is $T$-invariant, then since $\mu$ is uniquely ergodic, each limit point is equal to $\mu'$ and we are done. Hence it remains to show that each limit point is $T$-invariant. To do this, we show that for all $g \in C(X)$,
            \[
                \lim_{N \to \infty} \Big| \int_X g \, d\mu_N - \int_X g \circ T \, d\mu_N  \Big| = 0.
            \]
            By definition of the measures $\mu_N$, we need to show that for all $g \in C(X)$,
            \[
                \lim_{N \to \infty} \Big| \AVG F(\varphi(n)) \, g(T^{\Omega(n)}x) - \AVG F(\varphi(n)) \, g(T^{\Omega(n)+1}x)  \Big| = 0.
            \]
            Fix $g \in C(X)$. By \cref{prop: Invariance Property} applied when $a(n) = g(T^{\Omega(n)}x)$, we are done.  
        \end{proof}

        \begin{proof}[Proof of \cref{thm: Generalization of EK + PNT}]
            Let \( (X, \mu, T) \) be the uniquely ergodic system given by rotation on two points (See \cref{subsec: Number Theory} for the precise definition). Define \( g: X \to \{-1,1\} \) by \( g(0) = -1 \) and \( g(1) = 1 \). Then 
            \[
                \lambda(n) = g(T^{\Omega(n)}(0)). 
            \]
            By \cref{thm: Generalization of EK + Thm A}, combined with the Prime Number Theorem,
            \begin{align*}
                \lim_{N \to \infty} \AVG F \Big( \frac{\Omega(n) - \log \log N }{\sqrt{\log \log N}} \Big) \lambda(n) 
                &=
                \Big( \frac{1}{\sqrt{2\pi}} \int_{-\infty}^\infty F(t) e^{-t^2/2} \, dt \Big)\Big(\lim_{N \to \infty} \sum_{n=1}^{N} \lambda(n) \Big)
                \\
                &= 
                0.
            \end{align*}
        \end{proof}

    \subsection*{Acknowledgements} The author thanks Bryna Kra and Florian Richter for helpful discussions and feedback throughout this project.

\bibliography{MainReferenceLibrary}{}
\bibliographystyle{halpha}

\end{document}